\documentclass[12pt]{amsart}
\usepackage{graphicx} 
\usepackage[utf8]{inputenc}
\usepackage{amsthm, tikz, tikz-cd, bbm} 
\usetikzlibrary{arrows}
\usepackage{amssymb, amsfonts}
\usepackage{amsmath}
\usepackage{pdfpages}
\usepackage{mathrsfs}
\usepackage{enumerate}
\usepackage[all]{xy}
\usepackage{xcolor}
\usepackage{tabularx}

\usepackage{hyperref}

\theoremstyle{plain}
\newtheorem{prop}{Proposition}
\newtheorem{theorem}[prop]{Theorem}
\newtheorem{lemma}[prop]{Lemma}
\newtheorem{corollary}[prop]{Corollary}
\newtheorem{definition}[prop]{Definition}

\newtheorem{remark}[prop]{Remark}
\newtheorem{fact}[prop]{Fact}
\newtheorem{example}[prop]{Example}

\newcommand{\ZZ}{\mathbb{Z}}
\newcommand{\spec}{\operatorname{Spec}}
\newcommand{\kbar}{\overline{k}}
\newcommand{\Xbar}{\overline{X}}
\newcommand{\PP}{\mathbb{P}}
\newcommand{\pic}{\mathrm{Pic}}
\newcommand{\aut}{\operatorname{Aut}}

\newcommand{\Gal}{\mathrm{Gal}}
\newcommand{\Frob}{\mathrm{Frob}}

\usepackage{lipsum}
\makeatletter
\renewcommand{\subsection}{\@startsection{subsection}{2}%
  \z@{.8\linespacing\@plus.2\linespacing}{.5\linespacing}%
  {\normalfont\bfseries}}
\makeatother

\title[Proportion of Conic Bundles That are Rational]{What are the odds that a conic bundle over a finite field is rational?}
\author{Amanda Hernandez}
\email{amanda\underline{ }hernandez@brown.edu}
\address{Department of Mathematics, Brown University, Providence, Rhode Island, USA}
\date{July 13, 2026}

\begin{document}

\maketitle

\begin{abstract}

  We study rationality of conic bundles  over finite fields defined over the line. In particular, for a fixed number $n$ of degenerate fibers, what is the proportion of good conic bundles that are rational over a finite field of order $q$? The Galois action on their Picard groups factors through the Weyl group $W(D_n)$, reducing rationality to a condition on signed permutation types. Using work of Colliot-Thélène and Yang together with Chebotarev density, we compute the asymptotic proportion of such conic bundles that are rational as $q$ goes to infinity.
\end{abstract}

\section{Introduction}
An algebraic variety $X$ over a field $k$ is said to be \textbf{rational}
if it is birationally equivalent over $k$ to projective space $\mathbb{P}^n_k$ 
for some $n$. Similarly, it is \textbf{stably rational} if $X \times \mathbb{P}^m$ is rational for some $m \ge 0$. The question of whether a given variety is rational remains one of the central questions in birational geometry.

In this paper we study smooth projective surfaces  $X$ with $
\pi : X \to \mathbb{P}^1$ with generic fiber a smooth conic.
Such a surface is called a \textbf{conic bundle}. We will call it \textbf{good} if all the fibers are smooth conics or conics with a single singularity. The geometry of conic bundles is largely related to the number $n$ of degenerate (non-smooth) fibers.

The primary theorem of the paper describes the asymptotic proportion of these good conic bundles defined over $\mathbb{F}_q$ as $q \rightarrow \infty$.

\begin{theorem}\label{thm:main}
    As $q \rightarrow \infty$ ($q$ odd), the proportion of good conic bundles defined over $\mathbb{F}_q$ with $n$ degenerate fibers that are rational is asymptotic to \[ \frac{\binom{n}{2}(2n-5)!!+ 2n\binom{n-1}{2}(2n-7)!! + (2n-1)!!}{2^{n-1}n!} .\] As $n \rightarrow \infty$, this proportion is asymptotic to $\frac{2.5}{\sqrt{\pi n}}$.
\end{theorem}

Let $k$ be a finite field of characteristic $\neq 2$, and let
$X \to \mathbb{P}^1$ be a conic bundle defined over $k$ with
$n$ degenerate fibers. The absolute Galois group $\Gal ( \kbar / k)$ acts on
$\operatorname{Pic}X_{\kbar}$, and this action factors
through the Weyl group $W(D_n)$. As we will see below, the rationality of these conic bundles is determined by this action and can be checked via a cohomological condition. Thus the birational geometry of $X$ is encoded by a conjugacy class
in $W(D_n)$.

A fundamental result of Colliot-Thélène \cite{CTstablerat} shows that
over a finite field, stable rationality and rationality are equivalent
for smooth projective surfaces.
Moreover, stable rationality is implied by the vanishing of certain
Galois cohomology groups:
\[
H^1(\Gal(k'/k), \pic X) = 0
\]
for all finite extensions $k'/k$.
For conic bundles, this condition can be translated into a
cohomological condition on cyclic subgroups of $W(D_n)$.

In \cite{yang2024}, Yang studied the stable rationality of these surfaces over general non-closed fields. Looking specifically at cyclic subgroups, he found three types of actions that result in cohomology satisfying the necessary condition for stable rationality. Building on this work, which classifies the Weyl group elements satisfying this condition, we reduce the rationality problem to a purely combinatorial question: Which conjugacy classes of $W(D_n)$ satisfy this condition, and how frequently do they occur?

In his 2024 thesis, \cite{spaulding} studies a similar proportion for del Pezzo surfaces defined over a number field where, in that setting, $\pic \Xbar$ factors through $W(E_n)$. He assigns a probability density quantity to each subgroup of of $W(E_n)$.

Returning to the setting of a finite field $k$, we identify $\pic \Xbar$ with $\ZZ^{n+2}$ for conic bundles with $n$ degenerate fibers. Here we study the vanishing of the groups 
\[ \text{H}^1(\langle g^i \rangle , \ZZ^{n+2})\]
where $g \in W(D_n)$ and $1 \le i \le \text{ord}(g)$ for $i \vert \text{ord}(g)$. Using Yang's classification of actions satisfying the vanishing condition, we compute the probability of each type of action occurring for a conic bundle with $n$ degenerate fibers. 

We use a Chebotarev density result to establish equidistribution in an appropriate moduli space along with the classification of conjugacy classes in $W(D_n)$ to compute the asymptotic proportion of conic bundles over $\mathbb{F}_q$ with $n$ degenerate fibers that are rational as $q \to \infty$. In particular, in characteristic not 2, Theorem 1 tell us that the probability that a conic bundle with many
degenerate fibers is rational tends to zero at the rate $n^{-1/2}$.

Now there is some subtlety in the way that the order $q$  tends to infinity. One choice would be to fix a characteristic $p$ and take fields of order $q = p^r$ and let $r \rightarrow \infty$. The other would be to let the characteristics change and take increasingly large primes as the orders of our finite fields. We will demonstrate via different uses of the Chebotarev Density Theorem that the result holds in both cases.

\subsection*{Acknowledgements}
We are grateful to Anthony Várilly-Alvarado for conversations about this project. The author received partial support from U.S. Department of Education Graduate Assistance in Areas of National Need (GAANN) award P200A240015-25.

\section{Background}

\subsection{Conic Bundles}

We restrict our attention to a specific class of conic bundles, which we call good and balanced. These constructions are discussed in more detail in our upcoming article \cite{Hassett-Hernandez}.

Throughout, let $k$ be a field of characteristic not equal to $2$. 
\begin{definition}
A \textbf{good conic bundle} over $\PP^1$ consists of a 
smooth projective surface $X$ and a dominant morphism
$\phi: X \longrightarrow \PP^1$ such that
\begin{enumerate}
\item{the geometric generic fiber of $\phi$ is $\PP^1$}, and
\item{non-smooth fibers are isomorphic to two 
copies of $\PP^1$ meeting in a node.}
\end{enumerate}
\end{definition}

\begin{definition}
    Consider a good conic bundle $\phi:X \rightarrow \PP^1$ with
$$(\phi_*\omega_{\phi}^{-1})^{\vee} 
\simeq \mathcal{O}_{\PP^1}(a_0)  \oplus
\mathcal{O}_{\PP^1}(a_1) \oplus \mathcal{O}_{\PP^1}(a_2), \quad a_0 \le a_1 \le a_2.$$
We call the triple $(a_0,a_1,a_2) \in \mathbb{Z}^3$ the \textbf{splitting type} of the bundle. Moreover, we say the conic bundle is
\textbf{balanced} if $a_2-a_0 \le 1$.  
\end{definition}

Note that the balance condition is an open condition on conic bundles. In particular, if we restrict to good, balanced conic bundles we will be able to construct a sampling space that is of finite type when we look at the moduli space parameterizing such objects. 

Over algebraically closed fields, representative balanced good conic bundles may be obtained via blowing up
$\beta: X \longrightarrow \PP^1 \times \PP^1$. Examples of these constructions can be found in \cite{Hassett-Hernandez}.

The points of $\mathbb{P}^1$ over which the fiber is singular are called
\textbf{points of bad reduction}. We denote by $n$ the number of such
degenerate fibers. This number plays a central role in determining
the geometry of $X$, in particular the structure of its Picard group.

Let $\phi: X \longrightarrow \PP^1$ be a good conic bundle over an algebraically closed field with $n \ge 1$ degenerate fibers. Let $f$ be the class of a
fiber of $\phi$ and take $\Sigma$ to be a section of $\phi$. Write $E'_1,E''_1,\ldots,E'_n,E''_n$ for the irreducible components of the degenerate fibers so that
$$f \equiv E'_1 + E''_1 \equiv \cdots \equiv E'_n+E''_n;$$
we label so that
$$\Sigma \cdot E'_i=1, \Sigma \cdot E''_i=0, \quad i=1,\ldots,n.$$
We have
$$(E'_i)^2=(E''_i)^2=-1, \quad E'_i\cdot E''_i=1,$$
with the other intersection numbers between $\{E'_1,\ldots, E''_n\}$ equal to zero.  Note that $\{f,\Sigma,E''_1,\ldots,E''_n\}$ freely generates
$\pic X$ and 
$$K_X = -(r+2)f -2\Sigma + E''_1 + \cdots + E''_n,
\quad 
\Sigma^2=-r.$$

\begin{definition}
    Let $\phi: X \longrightarrow \PP^1$ be a good conic bundle with $n$ degenerate fibers. We say a collection of disjoint fibral curves $\{E_1', \ldots, E_n'\}$ is even if, after blowing down each $E_i'$ ,the ruled surface that results is a Hirzebruch surface $\bold{F}_r$ for some even $r$.
\end{definition}

Endowing this Picard group with the usual intersection form, we may examine the structure of the resulting group. This comes from a classical result \cite{SkorobogatovBIRS,HassettRSNF}) on intersections:
\begin{prop}[\cite{Hassett-Hernandez}, Proposition 15]
Consider $\pic X$ as a unimodular lattice under the intersection form.
For $n\ge 2$, the set
$$
R_n:=\left< K_X,f\right>^{\perp} \subset \pic X
$$
is isomorphic to the root lattice for $D_n$ and its extension
$$W_n:=\pic X /\left<K_X,f\right>$$
is isomorphic to the weight lattice.
The group
$$\{ g \in \aut(\pic X): g(f)=f, g(K_X)=K_X\}$$
equals the Weyl group $W(D_n)$.
\end{prop}

Thus, in order to understand the rationality of conic bundles over a fixed finite field $k$, we will examine the group structure of Weyl groups of type $D_n$. In particular, we are interested in the conjugacy class structure in order to study the cohomological results derived from \cite{CTstablerat}.

\subsection{Conjugacy Classes of the Weyl Group of Type D}

In Proposition~2.2 we saw that the orthogonal complement
\[
R_n := \langle K_X, f \rangle^\perp \subset \operatorname{Pic}(X)
\]
is isomorphic to the root lattice of type $D_n$, and that the
group of automorphisms of $\pic X$ fixing $f$
and $K_X$ is naturally identified with the Weyl group $W(D_n)$.
Thus the Galois action on $\pic \Xbar$
factors through $W(D_n)$.

To understand rationality via this Galois action, we therefore
require an explicit description of the conjugacy classes of $W(D_n)$.
We begin by recalling the structure of the Weyl groups of type $B_n$
and $D_n$, which may be described in terms of signed permutations.

The Weyl group of type $B_n$ is the group of all signed permutations
on $\{1, \dots, n\}$. Concretely, it consists of permutations of the
set $\{\pm 1, \dots, \pm n\}$ such that $g(-i) = -g(i)$ for all $i$.
It admits the description
\[
W(B_n) \cong S_n \ltimes (\mathbb{Z}/2\mathbb{Z})^n,
\]
Here, the symmetric group acts on the indices $\{ j^+, j^-\}$ by sending $j^\pm$ to $\sigma(j)^\pm$ and $c_j$ interchanges $j^+$ and $j^-$ and fixes all other indices. Intuitively, the permutation portion records how the irreducible
components of degenerate fibers are permuted, while the sign
records whether the two components of a given fiber $i$ are exchanged.
Thus the signed cycle type encodes the essential combinatorial
data of the Galois action on $\pic X$.

A signed permutation cycle $\alpha$ in this group takes the form $\alpha = c_{t_1} \cdots c_{t_s} \tau$ where $\tau \in S_n$ and the $c_{t_i}$ are all indices in $\tau$. General signed permutations are products of such cycles. 

We define $W(D_n)$ as the index two subgroup 
consisting of elements with an even number of sign changes.
Equivalently, it may defined as the kernel of the character $\sigma: W(B_n) \longrightarrow \{\pm 1\}$ defined by  \[\sigma(c_{t_1} \cdots c_{t_s} \tau) = (-1)^s.\] Thus, a signed permutation in $W(D_n)$ is the product of signed cycles $c_{t_1} \cdots c_{t_s} \tau$ with $s$ even. We define the projection map onto $S_n$
\[ pr: W(B_n) \longrightarrow S_n\]
by $pr(c_{t_1} \cdots c_{t_s} \tau) = \tau$. 

 Observe that it is always possible to write a signed permutation in $W(B_n)$ as a product of disjoint signed cycles $c_{t_1} \cdots c_{t_s} \tau$ where the $c_{t_i}$ only appear in their respective $\tau$ cycle, where we do allow $\tau$ to be trivial. Note that the total number of $c_{j}$ terms that appear in this expression must be even for the element to belong to $W(D_n)$.
 
 The order of $W(D_n)$ is
\[
|W(D_n)| = 2^{n-1} n!,
\]
a fact that will be used later when computing probabilities.

To understand the conjugacy class structure of $W(D_n)$, one first considers the conjugacy class structure of $W(B_n)$, the group of all signed permutations.

Associated to each permutation in $W(B_n)$ is its \emph{signed cycle type}, a sequence of the lengths of each cycle in its unique representation as a product of disjoint signed cycles where each length is assigned a positive sign or a negative sign depending on whether each cycle's character is positive or negative. In this group, two signed permutations are in the same conjugacy class if and only if they have the same signed cycle type.

Moreover, we record a fact from Carter about when two elements of $W(D_n)$ lie in the same conjugacy class.
\begin{prop}{\cite[Proposition 25]{Carter}}
An element of $W(B_\ell)$ lies in $W(D_\ell)$ if and only if it has an even number of negative terms in its signed cycle type. Two elements of $W(D_\ell)$ are conjugate if and only if they have the same cycle type, except that if all the cycles are even and positive there are two conjugacy classes.
\end{prop}

\begin{example}
    Let $n=4$. Representatives of the conjugacy classes of $W(D_4)$ may be chosen as follows:
\begin{align*}
&(1234), \quad c_1c_2(1234), \\
&(123)(4), \quad c_1c_2(123)(4), \\
&(12)(34), \quad c_1c_2(12)(34), \quad c_1c_3(12)(34), \\
&(12)(3)(4), \quad c_1c_2(12)(3)(4), \quad c_1c_3(12)(3)(4), \\
&(1)(2)(3)(4), \quad c_1c_2(1)(2)(3)(4), \quad
c_1c_2c_3c_4(1)(2)(3)(4).
\end{align*}
\end{example}

Since the cycles and $c_j$ terms represent how the Galois action permutes the irreducible components of the degenerate fibers, these conjugacy class types will play a major role in rationality of the respective conic bundle. In particular,
negative cycles correspond to fibers whose components are interchanged
by the Galois action. In the next section, we will see how these negative cycles are crucial in classifying which action types are permitted for rational conic bundles.

\subsection{Rationality Criteria}
We now explain how rationality of conic bundles over finite fields
may be detected via Galois cohomology, and how this reduces the
problem of interest to a condition on cyclic subgroups of $W(D_n)$.

We recall the $k$-birational classification of smooth, projective, geometrically rational $k$-surfaces. 
\begin{theorem}{\cite[Theorem 1]{iskovskih1979}}
    Let $k$ be a field and $X$ a smooth, projective, geometrically rational surface over $k$. Then there exists a $k$-birational morphism $X \longrightarrow Y$  where $Y$ is a smooth, projective $k$-minimal surface of one of the following types: \begin{enumerate}
        \item a relatively minimal conic bundle over a smooth conic
        \item a del Pezzo of degree $d$ with $1 \leq d \leq 9$.
    \end{enumerate}
\end{theorem}
Starting with a smooth, projective, geometrically rational surface, we can blow down Galois orbits consisting of disjoint exceptional curves to achieve one of the above two types of minimal models. Thus, to study rationality of geometrically rational surfaces, it suffices to analyze these two minimal types.

One can produce non-minimal conic bundles from relatively
minimal ones by blowing up points on smooth fibers.
If
\[
\phi:X \longrightarrow \mathbb P^1
\]
has $n$ degenerate fibers and the Frobenius action is encoded via an element $g \in W(D_n)$, then blowing up $m$ points on smooth
fibers or a length $m$ subscheme with at most one point in each fiber produces a good conic bundle
\[
\Tilde{\phi}: \Tilde{X} \longrightarrow \mathbb P^1
\]
with $n+m$ degenerate fibers. 

Looking at this from a group theoretic lens, we have tacked on a positive $m$-cycle with that is disjoint from the original $n$ fibers. Thus the Frobenius element for the resulting conic bundle is given by $h = g \cdot \tau \in W(D_{n+m})$ where $\tau$ is a positive cycle in the variables $\{n+1,...,n+m\}$ (i.e.\,it's only cyclically permuting the newly obtained fibers, but not interchanging the components). The resulting Picard group is constructed as the direct sum of the original Picard group and a permutation module. Thus, the cohomology remains unchanged. 

\begin{theorem}{\cite{CTstablerat}}
    Let $k$ be a field and $X$ a smooth, projective, geometrically rational $k$-surface.  Let $\Xbar = X \times_k \overline{k}$ and $\overline{k}$ be the algebraic closure of $k$. Assume $X(k) \neq 0$ and $X$ is split by a cyclic extension of $k$. If $X$ is not $k$-rational, then there exists a finite field extension $k'/k$ such that
    \[ H^1(\Gal(k'/k), \pic \Xbar) \neq 0\] and the k-variety X is not stably k-rational.
\end{theorem}


For conic bundles over finite fields, rationality and stable rationality are equivalent and the above cohomological vanishing condition is both necessary and sufficient to check rationality. We will say that $X$ satisfies condition $(*)$ if it satisfies the hypotheses of the theorem above. 


Since the absolute Galois group of a finite field is
topologically generated by the Frobenius automorphism,
the Galois action on $\pic \Xbar$
is determined by a single element $g \in W(D_n)$, well-defined up to conjugacy.
Thus the rationality problem reduces to a condition
on the cyclic subgroup generated by $g$.

Using the identification $
\pic\Xbar \cong \mathbb{Z}^{n+2}$, 
we are led to study the vanishing of
\[
H^1(\langle g^i \rangle, \mathbb{Z}^{n+2})
\quad \text{for all } 1 \le i \le \text{ord}(g), \quad i \vert \text{ord}(g).
\]

In other words, rationality of a conic bundle over $\mathbb{F}_q$
is equivalent to a purely group-theoretic condition on the action of a cyclic subgroup of $W(D_n)$ on $\mathbb{Z}^{n+2}$.  In particular, if we restrict to $k$ a finite field, the cohomological vanishing condition imposes constraints on the signed cycle structure of $g$ for $\langle g \rangle \subseteq W(D_n)$. 

The following result from Yang gives all possible signed cycle structures for generators of cyclic subgroups of $W(D_n)$ satisfying the $(*)$ condition.

\begin{prop}{\cite[Corollary 9]{yang2024}}\label{prop:yang}
    
If a cyclic group $G = \langle g \rangle$ satisfies the (*) condition, that is
\[\text{H}^1(\langle g^i \rangle , \pic\Xbar) = 0 \; \text{for all} \; i, \] and $\beta$ is a signed permutation cycle in $g$ with $\sigma(\beta) = -1$, then the length of the cycle $pr(\beta) \leq 2$ and there is at most one such signed permutation cycle. \\
Up to conjugation, there are three types of generators $g$ satisfying the assumption: \begin{enumerate}
    \item $g = c_1 c_2 \beta_3 \cdots \beta_R$,
    \item $g = c_1 c_2(2 3) \beta_3 \cdots \beta_R$,
    \item $g = \beta_1 \cdots \beta_R$,
\end{enumerate}
where $\sigma(\beta_i) = 1$ for all $i$.
\end{prop}

Consequently, rationality of a conic bundle over a finite field is determined entirely by whether the Frobenius element in $W(D_n)$ lies in one of these three families of conjugacy classes. Notably, large negative cycles, where the two irreducible components of some degenerate fiber are exchanged within a cycle of large length, obstruct a conic bundle from being rational.

This reduction transforms a geometric rationality problem
into a purely combinatorial  problem in the Weyl group $W(D_n)$. In the next section we compute the proportion
of elements of each type of possible rational conjugacy class.  

\begin{remark}
As of yet, this is merely computing the proportion of such conjugacy classes appearing in $W(D_n)$ as some conjugacy class representative may not, in fact, arise from a conic bundle in general. However, we will later see that equidistribution results show that this proportion represents the true proportion of rational conic bundles that we are interested in.
\end{remark}

\subsection{Moduli Space of Good Conic Bundles}\label{section:2.4}
The moduli-theoretic framework and monodromy results used here are developed in our forthcoming work \cite{Hassett-Hernandez}. We briefly recall the parts of this construction that are needed for the counting argument.

We introduce this moduli space to make precise what it means to choose a ``random" good balanced conic bundle with $n$ degenerate fibers. Rather than counting equations directly, we work with a parameter space whose points correspond to good balanced conic bundles defined over a finite field together with enough auxiliary data to keep track of the components of the degenerate fibers. The balance condition is useful here because it restricts attention to a finite type family, so that point counting and Chebotarev density arguments can be applied.

The monodromy comes from moving around in this parameter space and following the induced action on the Picard group of the fiber. Since the components of the degenerate fibers may be permuted, and the two components of a degenerate fiber may be exchanged, this monodromy naturally acts through the signed permutation group $W(D_n)$. 

\begin{prop}
 Good balanced conic bundles with $n$
degenerate fibers may be parametrized by a smooth irreducible scheme
of finite type. The generic such $X$ is isomorphic to a blowup of the Hirzebruch surface $\bold{F}_0$.
\end{prop}

\begin{definition}
    Let $\mathcal{C}_n^b$ be the moduli stack parametrizing
$$(\phi:X \longrightarrow \PP^1; E'_1,\ldots,E'_n)$$
where $\phi$ is a good balanced conic bundle
with $n$ degenerate fibers and
$E'_1,\ldots,E'_n$ is an even tuple of 
components of degenerate fibers. 
Isomorphisms are commutative diagrams
$$
\xymatrix{ X \ar^{\sim}[r] \ar_{\phi}[d] & Y \ar^{\varphi}[d] \\
\PP^1  \ar^{\sim}[r] & \PP^1
}
$$
where the horizontal arrows are isomorphisms
respecting the fibral curves and their order.
\end{definition}

\begin{prop}{\cite[Corollary 29]{Hassett-Hernandez}}
    The monodromy representation on Picard
groups of good conic bundles
with $n \ge 3$ degenerate fibers
is the Weyl group $W(D_n)$.  
\end{prop}

\begin{prop}{\cite[Corollary 40]{Hassett-Hernandez}}
    For $n \geq 4$, the moduli stack of balanced conic bundles $ \left[ W(D_n) \backslash \mathcal{C}_n^b \right]$ is separated with quasi-projective coarse moduli space.
\end{prop}

These moduli results provide the geometric input for the Chebotarev density argument we will present
below. After identifying the monodromy group with $W(D_n)$, the distribution
of Frobenius classes in this family can be compared with the uniform distribution on conjugacy classes of the group $W(D_n)$.

\subsection{Chebotarev Density Theorem}
In section 2.3, we reduced checking the rationality of a conic bundle over $\mathbb{F}_q$ to checking a condition on the conjugacy class of the Frobenius element in the Weyl group $W(D_n)$. We want to be able to deduce that counting the proportion of such conjugacy classes in $W(D_n)$ will yield a number asymptotic to the true proportion of rational conic bundles over $\mathbb{F}_q$ with $n$ degenerate fibers.

To achieve this we will make use of two formulations of Chebotarev density theorems. The intuitive idea of this theorem is that Frobenius elements become equidistributed among all the conjugacy classes of a finite Galois group as the size of the base field approaches infinity.

\begin{definition}[Big O Notation]
    Recall that we say a function $f(x) = O(g(x))$ if there exists a positive real number $M$ so that $\vert f(x) \vert \leq M g(x)$ for all $x$ in its domain.
\end{definition}

We first address the case of fixed characteristic.

\begin{theorem}{\cite{meagher}}\label{thm:meagher}
Let $G$ be a finite group and let $X/ k$ be a quasi-projective variety of
dimension $d$ that is geometrically irreducible, integral, and separated,
defined over a finite field $k$ of order $q = p^r$ for some prime $p$.
Assume that $G$ acts on $X$, and let $Y = X/G$ be the quotient and
$f : X \to Y$ the quotient morphism.
Assume further that $f$ is finite étale.

Let $C \subset G$ be a conjugacy class. Then
\[
\bigl|\{\, y \in Y_{\mathbb{F}_q}(\mathbb{F}_q) :
        \Frob_{q} \in C \,\}\bigr|
    \;=\;
    \frac{|C|}{|G|}\, q^{d}
    \;+\;
    O\!\left(q^{d - 1/2}\right),
\]
where $\mathbb{F}_q$ is an extension of $k$. 
\end{theorem}

We apply this theorem to the parameter space we construct as a subset of the moduli space of good conic bundles over $\mathbb{F}_q$ with $n$ degenerate fibers. We use this as the space from which we will randomly sample conic bundles, which correspond to $\mathbb{F}_q$-points in this moduli space, to determine what proportion are rational. 

In particular, dividing by $|Y(\mathbb{F}_q)|$ and using the Weil bounds
\[
|Y(\mathbb{F}_q)| = q^d + O(q^{d-1/2}),
\]
we obtain
\[
\frac{\#\{ y \in Y(\mathbb{F}_q) :
\mathrm{Frob}_q(y) \in C \}}
{|Y(\mathbb{F}_q)|}
=
\frac{|C|}{|G|}
+
O(q^{-1/2}).
\]

Thus, as $q \to \infty$ with characteristic $p$ fixed,
Frobenius elements become equidistributed among the
conjugacy classes of $G$. As a consequence, the asymptotic proportion of rational conic bundles with $n$ degenerate fibers is equal to the proportion of elements (conjugacy class generators) in $W(D_n)$ satisfying the cohomological vanishing condition.

For the mixed characteristic setting, we use the following formulation of Chebotarev's density theorem to achieve a similar result.

\begin{lemma}{\cite[Lemma 1.2]{ekedahl}}\label{thm:ekedahl}
Let $\pi : X \to \spec R$ be a morphism of finite type, where $R$ is a number ring.
Let $\rho : Y \to X$ be an étale Galois cover with Galois group $G$, and suppose that $\pi \circ \rho$ has geometrically irreducible generic fiber.
Let $C \subset G$ be a conjugacy class and set
\[
c := \frac{|C|}{|G|}.
\]
For each prime $\mathfrak{p} \in \text{max}(R)$,
let $t(\mathfrak{p})$ denote the number of
$x \in X(\kappa(\mathfrak{p}))$
for which the Frobenius element $\Frob_{\mathfrak{p}}(x)$
lies in $C$.
Then
\[
\frac{t(\mathfrak{p})}{|X(\kappa(\mathfrak{p}))|}
-
c
=
O\!\left(|\kappa(\mathfrak{p})|^{-1/2}\right).
\]
\end{lemma}

\begin{remark} The statement can be checked by passing to open subsets of X,Y compatible with $\rho$ at cost of changing the constant and error term. They will still be independent of q.
\end{remark}

\begin{prop}\label{prop:myprop}
Let $K$ be a number field with ring of integers $\mathcal{O}_K$. Let
\[
\rho : Y \to X
\]
be a generically finite Galois cover of normal quasiprojective varieties $X$ and $Y$ defined
over $K$ with $Y$ geometrically integral. Let $G$ be the Galois group of the cover. Let $C \subset G$
be a conjugacy class and set
\[
c := \frac{|C|}{|G|}.
\]

Then there exists a finite set $T \subset \text{max} \left(\mathcal{O}_K \right)$  such that for
all $p \notin T$,
\[
\frac{t(\mathfrak{p})}{|X(\kappa(\mathfrak{p}))|} - c
= O\!\left(N(\mathfrak{p})^{-1/2}\right),
\]
where
\[
t(p) := \#\{x \in X(\kappa(\mathfrak{p})) : \mathrm{Frob}_\mathfrak{p}(x) \in C\},
\]
$\kappa(p)$ denotes the residue field at $\mathfrak{p}$, $N(\mathfrak{p}) = |\kappa(\mathfrak{p})|$ is the
norm of $\mathfrak{p}$, and $\mathrm{Frob}_\mathfrak{p}(x)$ denotes the Frobenius conjugacy class
associated to $x$.
\end{prop}
\begin{proof}
Since $X$ and $Y$ are normal and quasiprojective over $K$, we may choose
projective models $\overline X$ and $\overline Y$ over $\mathcal O_K$ whose
generic fibers contain $X$ and $Y$ as dense open subschemes. Replacing these
models by their normalizations, we may assume that $\overline X$ and
$\overline Y$ are normal.

The morphism $\rho:Y\to X$ induces a rational map $
\overline \rho:\overline Y \dashrightarrow \overline X$.  Let $\Gamma\subset \overline Y\times_{\mathcal O_K}\overline X$ be the closure
of the graph of $\rho$, and let $\mathcal Y$ be the normalization of $\Gamma$.
Then $\mathcal Y$ is projective over $\mathcal O_K$, and the projection to
$\overline X$ gives a morphism
\[
\mathcal Y\to \overline X
\]
extending $\rho$ on the generic fiber. Replacing $\overline X$ by a
normal model $\mathcal X$ of $X$, we get a morphism
\[
\mathcal Y\to \mathcal X
\]
whose generic fiber is $\rho:Y\to X$. Since $\rho$ is generically finite, after
discarding finitely many primes of $\mathcal O_K$, we may assume that $\mathcal Y\to \mathcal X$ is generically finite.

Now let $R=\mathcal O_K[S^{-1}]$ for a finite set of primes $S$ large enough so
that $\mathcal X$ and $\mathcal Y$ are flat over $R$ and the above morphism is a
morphism of $R$-schemes. Since $\rho$ is a generically finite Galois cover with
group $G$, there exists a dense open subset $U\subset X$ such that, setting
$V:=\rho^{-1}(U)$, the morphism
\[
V\to U
\]
is a finite \'etale Galois cover with Galois group $G$. Thus Lemma~\ref{thm:ekedahl} applies to the finite \'etale Galois cover
\[
\mathcal V\to \mathcal U .
\]
Hence, for every prime $\mathfrak p\notin S$,
\[
\frac{t_{\mathcal U}(\mathfrak p)}
{|\mathcal U(\kappa(\mathfrak p))|}
-
c
=
O(N(\mathfrak p)^{-1/2}),
\]
where
\[
t_{\mathcal U}(\mathfrak p)
=
\#\{x\in \mathcal U(\kappa(\mathfrak p)):
\operatorname{Frob}_{\mathfrak p}(x)\in C\}.
\]
Since we can check Lemma ~\ref{thm:ekedahl} on open subschemes, this is sufficient.


\end{proof}

\begin{remark}
    From the results in section ~\ref{section:2.4}, we know there exists a geometrically connected $W(D_n)$-cover  of the moduli stack of uniform good conic bundles with
$n$ degenerate fibers. For each $\mathbb{F}_q$ point  the Frobenius automorphism determines a conjugacy class
\[
[\mathrm{Frob}_q(x)] \subset W(D_n).
\]
\end{remark}





\section{Probabilities}\label{section:3}

We now count the elements of \(W(D_n)\) that satisfy the rationality
condition described in Proposition~\ref{prop:yang}. Recall that such an
element has one of the following three signed cycle types:
\[
\begin{array}{ll}
\text{\emph{Type 1}:} & g=c_1c_2\beta_3\cdots\beta_R,\\[2mm]
\text{\emph{Type 2}:} & g=c_1c_2(23)\beta_3\cdots\beta_R,\\[2mm]
\text{\emph{Type 3}:} & g=\beta_1\cdots\beta_R,
\end{array}
\]
where every $\beta_i$ is a positive signed cycle. Thus, Type 3 has
only positive cycles, Type 1 has two negative fixed points, and Type 2
has one negative $2$-cycle together with one negative fixed point. The remaining cycles are positive in each case.

We will count the number of elements within each subtype that generate cyclic subgroups of $W(D_n)$ which satisfy the $(*)$ condition. This will involve the enumeration of positive signed permutations. 

\begin{definition}
    The \textbf{unsigned Stirling numbers of the first kind} ${ n\brack k}$ count the number of permutations with exactly $k$ cycles (counting fixed points as cycles of length one).
\end{definition}

These numbers satisfy the following useful identity.

\begin{fact} \label{fact:23}
$(x + (n-1))(x +(n-2)) \cdots (x+2)(x+1)x = \sum_{k=1}^n { n\brack k} x^k$
\end{fact}

Since we will overall be interested in asymptotic behavior of the proportions we will compute, we recall a standard result on the asymptotics of double factorials. 

\begin{prop}\label{prop:factorials}
    For $n$ even, 
    \[n!! \sim \sqrt{\pi n} \left( \frac{n}{e} \right)^{n/2}.\]
    For $n$ odd, 
    \[ n!! \sim \sqrt{2 n} \left( \frac{n}{e} \right)^{n/2} \]
\end{prop}

We will now use these facts for our enumeration, beginning with elements belonging to type 3. 

\subsection{Generators of Type 3}
We first count elements of Type 3, i.e.
\[
g = \beta_1 \cdots \beta_k
\]
where each cycle is positive.

Let $n$ be an integer and let $g = \beta_1 \cdots \beta_k \in W(D_n)$ where consisting of $k$ disjoint permutation cycles $\beta_i$, each with $\sigma(\beta_i) = 1$. 

Then there are ${ n\brack k}$ ways to write the underlying cycle decomposition in $S_n$. Then for each $\beta_i$ with length $\ell_i$, there are $2^{\ell_i -1}$ ways to assign it an even number of -1 signs (putting $c_i$ in front). So there are $2^{n -k}$ ways to assign the -1 signs for each $g$. 

\begin{lemma}
    The number of signed permutations in $W(D_n)$ where each cycle has an even number of $c_j$ terms associated to it is given by \[ \sum_{k=1}^n { n\brack k} 2^{n-k} = (2n-1)!!\]
\end{lemma}
\begin{proof}
    Plug in $x=\frac{1}{2}$ to the identify in Fact ~\ref{fact:23}
\end{proof}

Thus, we deduce that the number of elements of $W(D_n)$ belonging to a conjugacy class with a generator of Type 3 is precisely $(2n-1)!!$.

\begin{corollary}
    The number of Type 3 elements in $W(D_n)$ is $(2n-1)!!$.
\end{corollary}

\subsection{Generators of Type 1}
Now we count the number of type 1 elements in $W(D_n)$. Type 1 elements are of the form $g = c_1 c_2 \beta_3 \cdots \beta_R$.

There are $\binom{n}{2}$ choices for the $2$-cycle.
On the remaining $n-2$ elements we choose
a positive signed permutation, of which there are \[
(2(n-2)-1)!! = (2n-5)!!
\]
possibilities.

\begin{lemma}
The number of Type 1 elements in $W(D_n)$ is
\[
\binom{n}{2}(2n-5)!!.
\]
\end{lemma}

\subsection{Generators of Type 2}

Type 2 elements are of the form
\[
g = c_1 c_2 (23) \beta_3 \cdots \beta_R,
\].
Note this means there is an odd $2$-cycle and one sign change of another index.
There are $n$ ways to choose the sign change, $\binom{n-1}{2}$ ways to choose the $2$-cycle, 2 ways to then choose which of those two indices to have the corresponding $c_j$ term appear. Then there are $(2n-7)!!$ ways to have the remaining $n-3$ indices permuted via even signed cycles. 

\begin{lemma}
   There are $2n \binom{n-1}{2} (2n-7)!!$ elements of Type 2 in $W(D_n)$.
\end{lemma}

\section{Proof of Theorem}
We now combine the counting results of Section ~\ref{section:3} with the equidistribution
results of Section ~\ref{section:2.4} to prove Theorem ~\ref{thm:main}

From Section 3, the number of elements of each type in $W(D_n)$ is:

\begin{itemize}

\item Type 1:
\[
\binom{n}{2}(2n-5)!!.
\]
\item Type 2:
\[
2n\binom{n-1}{2}(2n-7)!!.
\]
\item Type 3:
\[
(2n-1)!!.
\]
\end{itemize}

Thus the total number of possible rational generators is
\[
(2n-1)!! + \binom{n}{2}(2n-5)!! + 2n\binom{n-1}{2}(2n-7)!!.
\]
Since $|W(D_n)| = 2^{\,n-1}n!$, the proportion of these elements in $W(D_n)$ is
\[
P_n = \frac{(2n-1)!! + \binom{n}{2}(2n-5)!! + 2n\binom{n-1}{2}(2n-7)!!}{2^{\,n-1}n!}.
\]

By Theorem ~\ref{thm:meagher} and Proposition ~\ref{prop:myprop} (Chebotarev density in fixed and mixed
characteristic), Frobenius elements associated to good conic bundles
become equidistributed in $W(D_n)$ as $q \to \infty$.

Therefore, as $q \to \infty$, the proportion of rational good conic bundles
with $n$ degenerate fibers is asymptotic to
\[
\frac{
(2n-1)!! + \binom{n}{2}(2n-5)!! + 2n\binom{n-1}{2}(2n-7)!!}{2^{\,n-1}n!}.
\]

Thus we have established the primary result of this paper Theorem ~\ref{thm:main}. By applying Lemma ~\ref{prop:factorials}, we get the following corollary.

 \begin{corollary}
     As $n \rightarrow \infty$ this proportion is asymptotic to $\frac{2.5}{\sqrt{\pi n}}$.
 \end{corollary}   
\section{Log Convexity}
Table ~\ref{tab:table} and Figure ~\ref{fig:fig1} suggest that the probabilities $P_n$
decrease as $n$ grows, but that they decrease more slowly for larger $n$.
This behavior is captured by log convexity: The successive ratios
\[
\frac{P_{n+1}}{P_n}
\]
form an increasing sequence.

\begin{definition}
    A sequence $\{ a_n\}$ is \textbf{log convex} if and only if it satisfies \[a_n^2 \leq a_{n-1} a_{n+1}\] for all $n$. Equivalently,
\[
\frac{a_{n+1}}{a_n} \geq \frac{a_n}{a_{n-1}}.
\]
\end{definition}

\begin{prop}
   The sequence $\{ P_n\}$ is log convex for $n \geq 3$.

\end{prop}
\begin{proof}

Recall that
\begin{align*}
    P_n &= \frac{(2n-1)!!+\binom n2(2n-5)!!+ 2n\binom{n-1}{2}(2n-7)!!}{2^{n-1}n!} \\
    &= \frac{(2n-1)!!}{2^{n-1}n!} + \frac{\binom n2(2n-5)!!}{2^{n-1}n!} + \frac{2n\binom{n-1}{2}(2n-7)!!}{2^{n-1}n!}
\end{align*}
So we can express our sequence as a sum of three sequences $P_n = Q_n + R_n + S_n$. It is known that the sum of log convex sequences is also log convex \cite{Liu}(see proposition 2.1). Thus, it suffices to check that each component sequence is log convex. 

We compute 
 \[ \frac{Q_{n+1}}{Q_n} =\frac{(2n+1)!!}{2^n(n+1)!}
\cdot
\frac{2^{n-1}n!}{(2n-1)!!}
=
\frac{2n+1}{2(n+1)}
=
1-\frac{1}{2(n+1)}.\]
As $n$ gets larger, this sequence increases. Therefore, $Q_n$ is log convex.

Now we compute the following:

\begin{align*}
    \frac{R_{n+1}}{R_n} &=
\frac{\binom{n+1}{2}(2n-3)!!}{2^n(n+1)!}
\cdot
\frac{2^{n-1}n!}{\binom n2(2n-5)!!}. 
\\
&= \frac{2n-3}{2(n-1)} 
=1-\frac{1}{2(n-1)}.
\end{align*}
As before, this sequence increases as $n$ increases. Therefore, $R_n$ is log convex for $n \geq 2$. 

Similarly, we compute
\begin{align*}
\frac{S_{n+1}}{S_n}
&= \frac{2(n+1)\binom n2(2n-5)!!}{2^n(n+1)!} \cdot
\frac{2^{n-1}n!}{2n\binom{n-1}{2}(2n-7)!!}\\
&= \frac{2n-5}{2(n-2)} = 1-\frac{1}{2(n-2)}.
\end{align*}
This sequence is also increasing as $n$ increases for $n \geq 3$, implying that $S_n$ is log convex for $n \geq 3$. We deduce the summed sequence $P_n$ is log convex for $n \geq 3$. 

\end{proof}

\section{Future Work}

One potential direction for future work is to make the error terms in the Chebotarev density arguments effective. In this paper, we use Chebotarev
density to show that Frobenius conjugacy classes become equidistributed in
$W(D_n)$, with error term of order $O(q^{-1/2})$ in fixed characteristic
and $O(N(\mathfrak p)^{-1/2})$ in mixed characteristic.  However, the
constants implicit in these estimates are not computed explicitly.

Determining these constants would require a more detailed analysis of the geometry of the relevant parameter spaces and their associated $W(D_n)$-covers. One would need to control the Betti numbers
or other cohomological invariants. It could be interesting to study some of the other potential parameter spaces discussed in \cite{Hassett-Hernandez} as they could provide sharper error estimates.   

This would give a quantitative estimate for how quickly the observed proportion of rational conic bundles approaches the limiting proportion $P_n$. This would be useful for comparing the asymptotic prediction with explicit computations over finite fields.

\section*{Data and Tabulations} 

\begin{table}[ht] 
\centering
\caption{Counts and proportions of elements satisfying $(*)$.}
\label{tab:my_table}
\small
\begin{tabular}{|c|r|r|c|}
\hline
\textbf{$n$} & \textbf{$|W(D_n)|$} & \textbf{\# satisfying $(*)$} & \textbf{Proportion} \\
\hline
3 & 24 & 24 & 1 \\ \hline
4 & 192 & 147 & 0.765625 \\ \hline
5 & 1920 & 1275 & 0.6640625 \\ \hline
6 & 23040 & 13770 & 0.59765625 \\ \hline
7 & 322560 & 177030 & 0.548828125 \\ \hline
8 & 5160960 & 2635605 & 0.5106811523 \\ \hline
9 & 92897280 & 44563365 & 0.4797058105 \\ \hline
10 & 1857945600 & 843242400 & 0.4538574219 \\ \hline
11 & 40874803200 & 17651333700 & 0.4318389893 \\ \hline
12 & 980995276800 & 404932703175 & 0.4127774239 \\ \hline
13 & 25505877196800 & 10101814898175 & 0.396058321 \\ \hline
14 & 714164561510400 & 272263848006150 & 0.3812340498 \\ \hline
\end{tabular}
\label{tab:table}
\end{table}

\begin{figure} 
    \centering
    \includegraphics[width=0.8\linewidth]{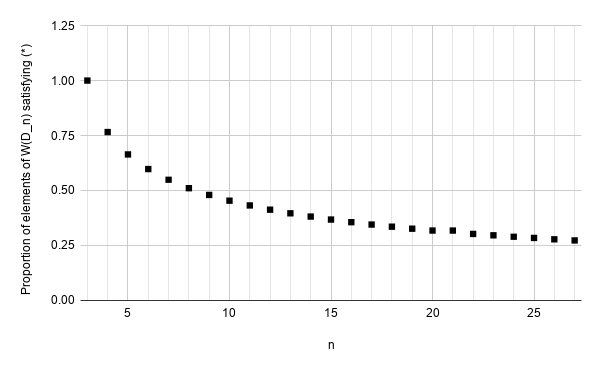}
    \caption{Proportion of elements in $W(D_n)$ satisfying $(*)$}
    \label{fig:fig1}
\end{figure}

\clearpage

\bibliography{biblio.bib}
\bibliographystyle{alpha}

\end{document}